\documentclass[10pt]{article}
\usepackage[dvips]{graphicx}
\usepackage{amssymb,color}
\usepackage[latin1]{inputenc}
\usepackage{epic}
\usepackage{pstricks}
\usepackage{pst-node}
\usepackage{pstricks-add}
\usepackage[latin1]{inputenc}
\usepackage{color}
\usepackage{lettrine}
\usepackage{calc,pifont}
\usepackage[noindentafter,calcwidth]{titlesec}
\usepackage{pstricks}
\usepackage{pst-plot}
\usepackage{pst-node}
\usepackage{amsmath}
\usepackage{amssymb}
\usepackage{amsthm}
\usepackage{subfigure}
\usepackage{lineno}

\def\BBox{\kern  -0.2cm\hbox{\vrule width 0.2cm height 0.2cm}}

\newtheorem{teo}{Theorem}[section]
\newtheorem{coro}[teo]{Corollary}
\newtheorem{lema}[teo]{Lemma}

\newtheorem{conjecture}[teo]{Conjecture}
\theoremstyle{definition}

\theoremstyle{remark}

\title{A note on extremal intersecting linear Ryser systems}
\author{Adrián Vázquez-Ávila\thanks{adrian.vazquez@unaq.mx}\\
{\small Subdirección de Ingeniería y Posgrado}\\
{\small Universidad Aeronáutica en Querétaro}\\
}

\date{}

\begin{document}
\maketitle
\begin{abstract}
A famous conjecture of Ryser states that any $r$-partite set system has transversal number at most $r-1$ times their matching number. This conjecture is only known to be true for $r\leq3$ in general, for $r\leq5$ if the set system is intersecting, and for $r\leq9$ if the intersecting set system is linear. In this note, we deal with Ryser's Conjecture for intersecting $r$-partite linear systems; that is, if $\tau$ is the transversal number for an intersecting $r$-partite linear system, then Ryser's Conjecture states that $\tau\leq r-1$. If this conjecture is true, this is known to be sharp for $r$ for which there exists a projective plane of order $r-1$. There has also been considerable effort to find intersecting $r$-partite set systems whose transversal number is $r-1$.  

In this note, the following is proved: if $r\geq4$ is an even integer, then $f_l(r)\geq3(r-2)+1$, where $f_l(r)$ is the minimum number of lines of an intersecting $r$-partite linear system whose transversal number is $r-1$. This lower bound gives an exact value for $f_l(r)$, for some small values of $r$. Also, we prove that any $r$-partite linear system satisfies $\tau\leq r-1$ if $\nu_2\leq r$ for all $r\geq3$ odd integer and $\nu_2\leq r-1$ for all $r\geq4$ even integer, where $\nu_2$ is the maximum cardinality of a subset of lines $R\subseteq\mathcal{L}$ such that every triplet of different elements of $R$ does not have a common point. 
\end{abstract}


\section{Introduction}
A \emph{set system} is a pair $(X,\mathcal{F})$ where $%
\mathcal{F}$ is a finite family of subsets on a ground set $X$. A
set system can be also thought of as a hypergraph, where the elements of $X$ and $\mathcal{F}$ are called \emph{vertices} and \emph{hyperedges} respectively. The set system $(X,\mathcal{F})$ is called \emph{$r$-uniform}, when all subsets of $\mathcal{F}$ has size $r$. The set system $(X,\mathcal{F})$ is \emph{$r$-partite} if the elements of $X$ can be partitioned into $r$ sets $X_1,\ldots,X_r$, called the \emph{sides}, such that each element of $\mathcal{F}$ contain exactly one element of $X_i$, for all $i=1,\ldots,r$, that is $|F\cap X_i|=1$, where $F\in\mathcal{F}$. Thus, an $r$-partite set system is an $r$-uniform set system.  

Let $(X,\mathcal{F})$ be a set system. A subset $T\subseteq X$ is a \emph{transversal} of $(X,\mathcal{F})$ if for any element $F\in\mathcal{F}$ contain at least an element of $T$, that is, $T\cap F\neq\emptyset$, for every $F\in\mathcal{F}$. The \emph{transversal number} of $(X,\mathcal{F})$, $\tau=\tau(X,\mathcal{F})$, is the smallest possible cardinality of a transversal of $(X,\mathcal{F})$. The transversal number has been studied in the literature in many different contexts and names. For example, with the name of \emph{piercing number} and \emph{co\-vering number}, see for instance \cite{AK06,AK06_2,AKMM01,Eckhoff,Huicochea,MS11,MR1921545,MR1149871,Oliveros}.

Let $(X,\mathcal{F})$ be a set system. A subset $\mathcal{E}\subseteq\mathcal{F}$ of a set system $(X,\mathcal{F})$ is called a \emph{matching} if $F\cap\hat{F}=\emptyset$ for every $F,\hat{F}\in\mathcal{E}$. The matching number of $(X,\mathcal{F})$, $\nu=\nu(X,\mathcal{F})$, is the cardinality of the largest matching of $(X,\mathcal{F})$. A set system is called \emph{intersecting} if $\nu=1$; that is, $F\cap\hat{F}\neq\emptyset$, for every $F,\hat{F}\in\mathcal{F}$.

It is not hard to see that any $r$-uniform set system $(X,\mathcal{F})$ satisfies the inequality $\tau\leq r\nu$. It is well-known this bound is sharp, as shown by the family of all subsets of size $r$ in a ground set of size $kr-1$, which has $\nu=k-1$ and $\tau=(k-1)r$. On the other hand, if $\nu=1$, any projective plane of order $r-1$ (denoted by $\Pi_{r-1}$ in this paper), where $r-1$ is a prime power, satisfies $\tau=r$. However, for $r$-partite set systems, Ryser conjectured in the 1960's that the upper bound could be improved.

{\bf Ryser's Conjecture: }Any $r$-partite set system satisfies $\tau\leq(r-1)\nu.$

For the special case $r=2$, Ryser's Conjecture is equivalent to K{\H o}nig's Theorem. The only other known general case of the conjecture is when $r=3$, see \cite{3Aharoni}. However, Ryser's conjecture is also known to be true in some special cases. Tuza verified \cite{<=5Tuza} Ryser's Conjecture for $r\leq5$ if the set system is assumed intersecting. Furthermore, Franceti{\' c} et al. \cite{FRANCETIC201791} verified Ryser's Conjecture for $r\leq9$ if the set system is a linear system, that is, a set system $(X,\mathcal{F})$ is \emph{linear} if it satisfies $|E\cap F|\leq 1$, for every pair of distinct subsets $E,F \in \mathcal{F}$. In this note every linear system is denoted by $(P,\mathcal{L})$, where the elements of $P$ and $\mathcal{L}$ are called \emph{points} and \emph{lines}, respectively. In the rest of this paper, only linear systems are considered. Most of the definitions can be generalized for set systems. Thus, we deal with Ryser's Conjecture for intersecting $r$-partite linear systems.

{\bf Intersecting linear Ryser's Conjecture:}
Every intersecting $r$-partite li\-near system satisfies $\tau\leq r-1$.
 
In case the conjecture would be true, it is tight in the sense that for infinitely many $r$'s there are constructions of intersecting $r$-partite linear systems with $\tau=r-1$. For example, if $r-1$ is a  prime power, consider the finite projective plane of order $r-1$ as a linear system, $\Pi_{r-1}$. This linear system is $r$-uniform and intersecting. To make it $r$-partite, one just needs to delete one point from the projective plane. This truncated projective plane gives an intersecting $r$-partite linear system with $\tau\geq r-1$, and $r(r-1)$ points and $(r-1)^2$ lines, and it will be denoted by $\Pi^\prime_{r-1}$. However, the construction obtained from the projective plane is not the ``optimal'' extremal. Although the projective plane construction only contains $r(r-1)$ points (which is an optimal number of points), it has a lot of lines. Let $f(r)$ be the minimum integer so that there exists an intersecting $r$-partite set system $(X,\mathcal{F})$ with $\tau= r-1$ and $|\mathcal{F}|=f(r)$ lines. Analogously, let $f_l(r)$ be the minimum integer so that there exists an intersecting $r$-partite linear system $(P,\mathcal{L})$ with $\tau=r-1$ and $|\mathcal{L}|=f_l(r)$ lines. It is probably that $f_l(r)$ does not exist for some values of $r$ (if $r-1$ is a prime power, then $\Pi^\prime_{r-1}$ is known to exist, providing proof that $f_l(r)$ is well-defined). Hence, if $f_l(r)$ does exist, for some $r$, then $f(r)\leq f_l(r)$. It is not difficult to prove that $f_l(2)=1$ and $f_l(3)=3$, see \cite{Mansour}. Furthermore, Mansour et al. \cite{Mansour} proved that $f_l(4)=6$ and $f_l(5)=9$. On the other hand, Aharoni et al. \cite{MultiAharoni} proved that $f_l(6)=13$ and $f(7)=17$ (even when the truncated projective plane does not exist, since it has been proved that finite projective planes of order six do not exist, see \cite{6projective_plane}); however Franceti{\' c} proved \cite{FRANCETIC201791} that $f_l(7)$ does not exist, that is, there is no intersecting $7$-partite linear system such that $\tau=6$. Abu-Khazneha et al. \cite{Abu-Khazneha} constructed new infinite family of intersecting $r$-partite set systems extremal to Ryser's Conjecture, which exist whenever a projective plane of order $r-2$ exists. That construction produces a large number of non-isomorphic extremal set systems. Finally, Aharoni et al. \cite{MultiAharoni} gave a lower bound on $f(r)$ when $r\to\infty$, showing that $f(r)\geq$3.052$r+O(1)$, this lower bound is an improvement since Mansour proved \cite{Mansour} that $f(r)\geq(3-\frac{1}{\sqrt{18}})r(1-o(1))\approx$2.764$(1-o(1))$, when $r\to\infty$. In this paper, a lower bound for $f_l(r)$ is given for $r\geq4$ an even integer. 

\begin{teo}\label{thm:main_intro}
If $r\geq4$ is an even integer, then $3(r-2)+1\leq f_l(r)$.
\end{teo}

Notice that, if $r\in\{2,4,6\}$, then $f_l(r)=3(r-2)+1$. Aharoni et al.  \cite{MultiAharoni} proved $18\leq f(8)$ and $24\leq f(10)$. Hence, by Theorem \ref{thm:main_intro} we have: $19\leq f_l(8)$ and $25\leq f_l(10)$.

\section{Main Results}\label{sec:previous}
Let $(P,\mathcal{L})$ be a linear system and $p\in P$ be a point. It is denoted by $\mathcal{L}_p$ to be the set of lines incident to $p$. The \emph{degree} of $p$ is defined as $deg(p)=|\mathcal{L}_p|$ and the maximum degree over all points of the linear system is denoted by $\Delta=\Delta(P,\mathcal{L})$. 

A subset $R$ of lines of a linear system $(P,\mathcal{L})$ is a \emph{$2$-packing} of $(P,\mathcal{L})$ if the elements of $R$ are triplewise disjoint, that is, if any three elements are chosen in $R$, then they have not a common point. The 2-packing number of $(P,\mathcal{L})$, $\nu=\nu_2(P,\mathcal{L})$, is the maximum cardinality of a 2-packing of $(P,\mathcal{L})$. There are some works that study this new parameter, see \cite{AvilaLetters,AvilaAKCE,MR3727901,Avila_covering_grphs,AvilaEnotes,AvilaArscomb,Avila_dom_gra,AvilaJDMSC,AvilaBSSM}. 

\begin{teo}\cite{AvilaAKCE}\label{thm:tau<=nu2-1}
Let $(P,\mathcal{L})$ be a linear system with $p,q\in P$ be two points such that $\Delta=deg(p)$ and $\Delta'=\max\{deg(x): x\in
P\setminus\{p\}\}$. If $|\mathcal{L}|\leq \Delta+\Delta'+\nu_2-3$, then $\tau\leq\nu_2-1$.
\end{teo}

\begin{lema}\cite{AvilaEnotes}\label{lemma:nu_2=r+1}
Let $(P,\mathcal{L})$ be an intersecting $r$-uniform linear system, with $r\geq3$ be an odd integer. If $\tau=r$, then $\nu_2=r+1$.
\end{lema}

\begin{coro}\label{Coro:nu=r,r_partita}
Let $(P,\mathcal{L})$ be an intersecting $r$-partite linear system, with $r\geq3$ be an odd integer. If $\nu_2\leq r$ then $\tau\leq r-1$.
\end{coro}


Now, let deal with the case when $r$ is an even integer. 

\begin{lema}\cite{AvilaArscomb}\label{lemma:impar}
Let $(P,\mathcal{L})$ be an $r$-uniform intersecting linear system with $r\geq2$ be an even integer. If $\nu_2=r+1$ then $\tau=\frac{r+2}{2}$. 
\end{lema}

Hence, if $\tau>\frac{r+2}{2}$ then $\nu_2\leq r$ and $r\geq4$. 

\begin{lema}\cite{AvilaArscomb}\label{lemma:nu_2=r}
Let $(P,\mathcal{L})$ be an intersecting $r$-uniform linear system, with $r\geq4$ be an even integer. If $\tau=r$, then $\nu_2=r$.
\end{lema}

\begin{lema}\label{lemma:nu_2=r,tau}
Let $(P,\mathcal{L})$ be an intersecting $r$-uniform linear system with $r\geq4$ be an even integer. If $\tau=r-1$, then $\nu_2=r$.
\end{lema}

\begin{proof}
Let $(P,\mathcal{L})$ be an intersecting $r$-partite linear system. Let $p,q\in P$ be two points such that $\Delta=deg(p)$ and $\Delta^\prime=\max\{deg(x): x\in
P\setminus\{p\}\}$. By Theorem \ref{thm:tau<=nu2-1} if $|\mathcal{L}|\leq\Delta+\Delta^\prime+\nu_2-3\leq3(r-2)+1$, then $\tau\leq\nu_2-1$, which implies that $\nu_2=r$. 
\end{proof}

\begin{coro}\label{coro:tau<r-2}
Let $(P,\mathcal{L})$ be an intersecting $r$-partite linear system with $r\geq4$ be an even integer. If $\nu_2\leq r-1$ then $\tau\leq r-2$.
\end{coro}

By Lemmas \ref{lemma:nu_2=r} and \ref{lemma:nu_2=r,tau}, we have:

\begin{coro}\label{coro:main}
Let $(P,\mathcal{L})$ be an intersecting $r$-uniform linear system with $r\geq4$ be an even number. If $\tau\in\{r-1,r\}$, then $\nu_2=r$.	
\end{coro} 

By Corollaries \ref{Coro:nu=r,r_partita} and \ref{coro:tau<r-2}, and Lemma \ref{lemma:nu_2=r}, we have:

\begin{teo}
Let $r\geq3$ be an integer. Then every intersecting $r$-partite linear system satisfies
\begin{enumerate}
	\item $\tau\leq r-1$ If $\nu_2\leq r$ and $r\geq3$ be an odd integer; and
	\item $\tau\leq r-2$ if $\nu_2\leq r-1$ and $r\geq4$ be an even integer.
\end{enumerate}
\end{teo}




The intersecting linear Ryser's Conjecture is almost proved; it only remains to analyzes two cases concerning the 2-packing number: 
\begin{conjecture}
Let $r\geq3$ be an integer. Then every intersecting $r$-partite linear system satisfies:
	\begin{enumerate}
		\item $\tau\leq r-1$ if $\nu_2=r+1$, with $r\geq3$ an odd number.
		\item $\tau\leq r-1$ if $\nu_2=r$, with $r\geq4$ an even number.
	\end{enumerate}	
\end{conjecture}

\begin{teo}
Let $r\geq4$ be an even integer, then $3(r-2)+1\leq f_l(r)$.
\end{teo}
\begin{proof}
Suppose that $\nu_2\leq r-1$ (by Corollary \ref{coro:main}). Let $p,q\in P$ be two points such that $deg(p)=\Delta$ and $\Delta'=\deg(q)=\max\{deg(x): x\in
P\setminus\{p\}\}$. By Theorem \ref{thm:tau<=nu2-1} if $|\mathcal{L}|\leq\Delta+\Delta'+\nu_2-3\leq3(r-2)$, then $\tau\leq\nu_2-1\leq r-2$. Therefore, $3(r-2)+1\leq f_l(r)$.
\end{proof}


{\bf Acknowledgment}

Research was partially supported by SNI and CONACyT.

\end{document}